\documentclass[11pt]{amsart}
\usepackage{verbatim}
\usepackage{amsmath}
\usepackage{enumerate}
\usepackage{amssymb}
\usepackage{color}
\usepackage{url}
\usepackage{graphicx}
\usepackage{fullpage}

\newcommand\+{\;\lower\plusheight\hbox{$+$}\;}
\newcommand\lldots{\;\lower\plusheight\hbox{$\cdots$}\;}
\newcommand{\SL}{\operatorname{SL}}
\newcommand{\Z}{\mathbb{Z}}

\newtheorem{Theorem}{Theorem}[section]
\newtheorem{Lemma}[Theorem]{Lemma}
\newtheorem{Corollary}[Theorem]{Corollary}

\newtheorem{Example}[Theorem]{Example}

\setbox0=\hbox{$+$}
\newdimen\plusheight
\plusheight=\ht 0 \setbox0=\hbox{$-$}
\newdimen\minusheight
\minusheight=\ht0

\setbox0=\hbox{$\cdots$}
\newdimen\cdotsheight
\cdotsheight=\plusheight
\numberwithin{equation}{section}
\allowdisplaybreaks

\begin{document}

\title{Analogues of the Ramanujan--Mordell Theorem}
\author{Shaun Cooper, Ben Kane and Dongxi Ye}
\address{Institute of Natural and Mathematical Sciences \\
Massey University-Albany \\
Private Bag 102904, North Shore Mail Centre\\
Auckland, New Zealand \\
E-mail: s.cooper@massey.ac.nz}
\address{Department of Mathematics\\
University of Hong Kong, Pokfulam, Hong Kong\\
E-mail address: bkane@maths.hku.hk}
\address{
Department of Mathematics, University of Wisconsin\\
480 Lincoln Drive, Madison, Wisconsin, 53706 USA\\
E-mail: lawrencefrommath@gmail.com}
\thanks{The research of the second author was supported by grant project numbers 27300314 and 17302515 of the Research Grants Council.}
\date{\today}
\begin{abstract}
The Ramanujan--Mordell Theorem for sums of an even number of squares is extended to other quadratic forms and
quadratic 
polynomials.
\end{abstract}
\maketitle

\section{Introduction}

One of the classical problems in number theory is to determine exact formulas for the number of representations of a positive integer $n$ as a sum of $2k$ squares, which we denote by $r(2k;n)$. If we set
\begin{equation}
\label{defz}
z=z(\tau):=\sum_{m=-\infty}^{\infty}\sum_{n=-\infty}^{\infty}q^{m^{2}+n^{2}}
\end{equation}
where, here and throughout the remainder of this work, $\tau$ is a complex number with positive imaginary part and $q=e^{2\pi i\tau}$, then 
(considering $z$ as a power series in $q$)
$$
\sum_{n=0}^{\infty}r(2k;n)q^{n}=z^{k}.
$$
The function $z^k$ is a modular form and it is well-known that
$$
{{
z^{k}(\tau)
}}
=E^{*}_{k}(\tau)+C_{k}(\tau),
$$
where $E^{*}_{k}(\tau)$ is an Eisenstein series and $C_{k}(\tau)$ is a cusp form.  In his remarkable work, Ramanujan \cite[Eqs. (145)--(147)]{ram} stated without proof 
explicit formulas for $E^{*}_{k}(\tau)$ and $C_{k}(\tau)$,
and hence deduced the value of the coefficients $r(2k;n)$.
Ramanujan's result was first proved by Mordell \cite{mordell}.
To state it we need Dedekind's eta function, which is defined by 
$$
{{
\eta(\tau):=
}}
q^{1/24}\prod_{j=1}^{\infty}(1-q^{n}).
$$
Here and throughout, we write $\eta_{m}$ for $\eta(m\tau)$ for any positive integer $m$. 
\begin{Theorem}[Ramanujan--Mordell]
\label{rammor}
Suppose $k$ is a positive integer. Let $z$ be defined by \eqref{defz}. Then
\begin{equation}
\label{zF1}
z^{k}=F_{k}(\tau)+z^{k}\sum_{1\leq j\leq\frac{(k-1)}{4}}c_{j,k}x^{j}
\end{equation}
where $c_{j,k}$ are numerical constants that depend on $j$ and $k$, 
$$
x=x(\tau):=\frac{\eta^{48}_{2}
}{\eta_{1}^{24}
\eta_{4}^{24}
},
$$
and $F_{k}(\tau)$ is an Eisenstein series defined by:
$$
F_{1}(\tau):=1+4\sum_{j=1}^{\infty}\frac{q^{j}}{1+q^{2j}},
$$
and for $k\geq1$,
$$
F_{2k}(\tau):=1-\frac{4k(-1)^{k}}{(2^{2k}-1)\mathcal{B}_{2k}}\sum_{j=1}^{\infty}\frac{j^{2k-1}q^{j}}{1-(-1)^{k+j}q^{j}},
$$
and
$$
F_{2k+1}(\tau):=1+\frac{4(-1)^{k}}{\mathcal{E}_{2k}}\sum_{j=1}^{\infty}\left(\frac{(2j)^{2k}q^{j}}{1+q^{2j}}-\frac{(-1)^{k+j}(2j-1)^{2k}q^{2j-1}}{1-q^{2j-1}}\right).
$$
Here $\mathcal{B}_{k}$ and $\mathcal{E}_{k}$ are the Bernoulli numbers and Euler numbers, respectively, defined by
$$
\frac{u}{e^{u}-1}=\sum_{k=0}^{\infty}\frac{\mathcal{B}_{k}}{k!}u^{k}\quad\mbox{and}\quad \frac{1}{\cosh{u}}=\sum_{k=0}^{\infty}\frac{\mathcal{E}_{k}}{k!}u^{k}.
$$
\end{Theorem}
The reader is referred to \cite[p.~2]{chancooper} for a brief account of 
the 
history of the study of Theorem~\ref{rammor} for various $k$.

The goal of this work is to prove the analogues of the Ramanujan--Mordell Theorem for which the quadratic form $m^{2}+n^{2}$ in \eqref{defz} is replaced with the quadratic form
$m^2+pn^2$, or by the quadratic 
polynomial
$$
\frac{m(m+1)}{2}+p\frac{n(n+1)}{2}, 
$$
where $p\in \left\{3,7,11,23\right\}$.

This work is organized as follows. In Section \ref{defmain}, we set up definitions, state the main results
and explicate several examples. Proofs of the main results
are given in Section \ref{proofs}.

\section{Definitions and Main Results}
\label{defmain}
For $k\geq1$, define the normalized Eisenstein series by
\begin{equation}
\label{E2k}
E_{2k}(\tau):=1-\frac{4k}{\mathcal{B}_{2k}}\sum_{j=1}^{\infty}\frac{j^{2k-1}q^{j}}{1-q^{j}}.
\end{equation}
Let $p$ be an odd prime. The generalized Bernoulli numbers $\mathcal{B}_{k,p}$ are defined by
\begin{equation}
\label{Bkp}
\frac{x}{e^{px}-1}\sum_{j=1}^{p-1}
{{
\chi_p(j)
}}
e^{jx}=\sum_{k=0}^{\infty}\mathcal{B}_{k,p}\frac{x^{k}}{k!},
\end{equation}
where $
{{
\chi_p(j):=
}}
\left(\frac{j}{p}\right)$ is the Legendre symbol. Let $k$ be a positive integer which satisfies
$$
k\equiv\frac{p-1}{2}\pmod{2}.
$$
The generalized Eisenstein series $E_{k}^{0}(\tau;\chi_{p})$ and $E_{k}^{\infty}(\tau;\chi_{p})$ 
at the cusps $0$ and $i\infty$, respectively, 
are defined by 
\begin{align*}
E_{k}^{0}(\tau;\chi_{p})&:=\delta_{k,1}-\frac{2k}{\mathcal{B}_{k,p}}\sum_{j=1}^{\infty}\frac{j^{k-1}}{1-q^{pj}}\sum_{\ell=1}^{p-1}\left(\frac{\ell}{p}\right)q^{j\ell}
\intertext{and}
E_{k}^{\infty}(\tau;\chi_{p})&:=1-\frac{2k}{\mathcal{B}_{k,p}}\sum_{j=1}^{\infty}\left(\frac{j}{p}\right)\frac{j^{k-1}q^{j}}{1-q^{j}},
\end{align*}
where $\delta_{m,n}$ is the Kronecker delta function, defined by
$$
\delta_{m,n}:=\begin{cases}1\quad\mbox{if $m=n$,}\\0\quad\mbox{if $m\ne n$.}\end{cases}
$$
Certain linear combinations of Eisenstein series and generalized Eisenstein series will occur
in the main results, and in anticipation of this we define series $G_{k}(\tau;p)$, $\tilde{G}_{2k+1}(\tau;p)$,
$F_{k}(\tau;p)$ and $\tilde{F}(\tau;p)$
as follows. For $k\geq 1$, let
$$
G_{2k}(\tau;p) := E_{2k}(\tau) + (-p)^k E_{2k}(p\tau).
$$
For $k\geq 0$ let
\begin{align*}
G_{2k+1}(\tau;p) &:= E_{2k+1}^\infty(\tau;\chi_p)+(-p)^kE_{2k+1}^0(\tau;\chi_p)
\intertext{and}
\tilde{G}_{2k+1}(\tau;p) &:= E_{2k+1}^\infty(\tau;\chi_p)-(-p)^kE_{2k+1}^0(\tau;\chi_p).
\end{align*}
For $p=3$ or $11$ and $k\geq0$, let
\begin{equation}
\label{Fo}
F_{2k+1}(\tau;p):= \frac{G_{2k+1}(\tau;p)+2^{2k+1}G_{2k+1}(4\tau;p)}{(2^{2k}+1)(1+\delta_{k,0})}
\end{equation}
and
\begin{equation}
\label{Fto}
\tilde{F}_{2k+1}(\tau;p):=\frac{(2^{2k+1}-2)\tilde{G}_{2k+1}(\tau;p)-2^{2k+1}G_{2k+1}(2\tau;p)+2G_{2k+1}(\tau/2;p)}
{2^{4k+2}(2^{2k+1}+1)(1+\delta_{k,0})}.
\end{equation}
For $p=7$ or $23$ and $k\geq0$, let
\begin{equation}
\label{Fo7}
F_{2k+1}(\tau;p):= \frac{G_{2k+1}(\tau;p)-2G_{2k+1}(2\tau;p)+2^{2k+1}G_{2k+1}(4\tau;p)}
{(2^{2k+1}-1)(1+\delta_{k,0})}
\end{equation}
and
\begin{equation}
\tilde{F}_{2k+1}(\tau;p) := \frac{G_{2k+1}(\tau;p)-G_{2k+1}(2\tau;p)}{2^{2k+1}(2^{2k+1}-1)(1+\delta_{k,0})}.
\end{equation}
For $p=3$, $7$, $11$ or $23$ and any integer $k\geq1$, let
\begin{equation}
\label{Fe}
F_{2k}(\tau;p) := \frac{G_{2k}(\tau;p)-2G_{2k}(2\tau;p)+2^{2k}G_{2k}(4\tau;p)}{(2^{2k}-1)(1+(-p)^k)}
\end{equation}
and
\begin{equation}
\label{Fte}
\tilde{F}_{2k}(\tau;p) := \frac{G_{2k}(\tau;p)-G_{2k}(2\tau;p)}{2^{2k}(2^{2k}-1)(1+(-p)^k)}.
\end{equation}
The observant reader will notice that the series $\tilde{G}_{2k+1}$ occurs only once in the definitions \eqref{Fo}--\eqref{Fte}, and that is as one of the terms in~\eqref{Fto}. { A reason for this will be seen later in Lemma \ref{Fhalf}, in which \eqref{E2k/2} and \eqref{E2k3} imply that for $p=3$ or $11$
$$
G_{2k+1}\left(\tau+\frac{1}{2};p\right)=-G_{2k+1}(\tau;p)+\left(2-2^{2k+1}\right)\tilde{G}_{2k+1}(2\tau;p)+2^{2k+1}G_{2k+1}(4\tau;p).
$$}

For $p=3,$ $7$, $11$ or $23$, let $z_{p}$ and $x_{p}$ be defined by
\begin{equation}
\label{zp}
z_{p}=z_{p}(\tau):=\sum_{m=-\infty}^{\infty}\sum_{n=-\infty}^{\infty}q^{m^{2}+pn^{2}},\\
\end{equation}
and 
\begin{equation}
\label{x3}
x_{p}=x_{p}(\tau):=\frac{\left(\eta_{1}\eta_{4}\eta_{p}\eta_{4p}\right)^{24/(p+1)}}{\left(\eta_{2}\eta_{2p}\right)^{48/(p+1)}}.
\end{equation}
The analogue of the Ramanujan--Mordell Theorem, and the main result of this work, is:
\begin{Theorem}
\label{main}
Suppose $p=3$, $7$, $11$ or $23$ and let $k$ be a positive integer. Then
$$
z_{p}^{k}=F_{k}(\tau;p)+z_{p}^{k}\sum_{1\leq j<\frac{(p+1)k}{8}}c_{p,k,j}x_{p}^{j},
$$
where $c_{p,k,j}$ are numerical constants that depend only on $p$, $k,$ and $j$.
\end{Theorem}
There is an analogue of the Ramanujan--Mordell theorem that involves sum of triangular numbers
instead of sums of squares; that is, replace
the quadratic form $m^2+n^2$ in \eqref{defz} and Theorem~\ref{rammor} with the quadratic 
polynomial 
$m(m+1)/2+n(n+1)/2$. See \cite[pp. 190--191]{ram} and \cite[Theorems 3.5 and~3.6]{sc}.
To state the corresponding analogue of Theorem~\ref{main}, let 
$\tilde{z}_{p}$ and $\tilde{x}_{p}$ be defined by
\begin{equation}
\label{tilz}
\tilde{z}_{p}=\tilde{z}_{p}(\tau):=q^{(p+1)/8}\sum_{m=0}^{\infty}\sum_{n=0}^{\infty}q^{\frac{m(m+1)}{2}+p\frac{n(n+1)}{2}}
\end{equation}
and 
\begin{equation}
\tilde{x}_{p}=\tilde{x}_{p}(\tau):=\left(\frac{\eta_{1}\eta_{p}}{\eta_{2}\eta_{2p}}\right)^{24/(p+1)}.
\end{equation}

\begin{Corollary}
\label{main2}
Suppose $p=3$, $7$, $11$ or $23$ and let $k$ be a positive integer. Then
$$
\tilde{z}_{p}^{k}=\tilde{F}_{k}(\tau;p)+\tilde{z}_{p}^{k}\sum_{1\leq j<\frac{(p+1)k}{8}}\tilde{c}_{p,k,j}\tilde{x}_{p}^{j},
$$
where $\tilde{c}_{p,k,j}$ are numerical constants that depend only on $p$, $k,$ and $j$ and are related to the $c_{p,k,j}$ in Theorem $\ref{main}$ by $\tilde{c}_{p,k,j}=2^{-24j/(p+1)}(-1)^{j}c_{p,k,j}$.
\end{Corollary}
%

In the remainder of this section we describe some special cases of Theorem~\ref{main} and Corollary~\ref{main2}.
Let $\varphi(q)$ and $\psi(q)$ be Ramanujan's theta functions defined by
$$
\varphi(q)=\sum_{n=-\infty}^{\infty}q^{n^{2}}\quad\mbox{and}\quad
\psi(q)=\sum_{n=0}^{\infty}q^{\frac{n(n+1)}{2}}.
$$
Because of the occurrence of the argument $\tau/2$ in one of the terms in \eqref{Fto},
we have replaced $\tau$ with $2\tau$, and hence $q$ with $q^2$, to obtain the examples \eqref{t3k1},
\eqref{t11k1} and \eqref{t3k3}, below.
\begin{Example}
For $k=1$ and $p=3,\,7,\,11$ or $23$, Theorem $\ref{main}$ and Corollary~$\ref{main2}$ give
\begin{align}
\label{s3k1}
\varphi(q)\varphi(q^{3})&=1+2\sum_{j=1}^{\infty}\left(\frac{j}{3}\right)\left(\frac{q^{j}}{1-q^{j}}+\frac{2q^{4j}}{1-q^{4j}}\right),\\
\label{t3k1}
q\psi(q^2)\psi(q^{6})&=\sum_{j=1}^{\infty}\left(\frac{j}{3}\right)\left(\frac{q^{j}}{1-q^{j}}-\frac{q^{4j}}{1-q^{4j}}\right),\\
\label{s7k1}
\varphi(q)\varphi(q^{7})&=1+2\sum_{j=1}^{\infty}\left(\frac{j}{7}\right)\left(\frac{q^{j}}{1-q^{j}}-\frac{2q^{2j}}{1-q^{2j}}+\frac{2q^{4j}}{1-q^{4j}}\right),\\
\label{t7k1}
q\psi(q)\psi(q^{7})&=\sum_{j=1}^{\infty}\left(\frac{j}{7}\right)\left(\frac{q^{j}}{1-q^{j}}-\frac{q^{2j}}{1-q^{2j}}\right),\\
\label{s11k1}
\varphi(q)\varphi(q^{11})&=1+\frac{2}{3}\sum_{j=1}^{\infty}\left(\frac{j}{11}\right)\left(\frac{q^{j}}{1-q^{j}}+\frac{2q^{4j}}{1-q^{4j}}\right)+\frac{4}{3}\eta_{2}\eta_{22},\\
\label{t11k1}
q^{3}\psi(q^{2})\psi(q^{22})&=\frac{1}{3}\sum_{j=1}^{\infty}\left(\frac{j}{11}\right)\left(\frac{q^{j}}{1-q^{j}}-\frac{q^{4j}}{1-q^{4j}}\right)-\frac{1}{3}\eta_{2}\eta_{22},\\
\label{s23k1}
\varphi(q)\varphi(q^{23})&=1+\frac{2}{3}\sum_{j=1}^{\infty}\left(\frac{j}{23}\right)\left(\frac{q^{j}}{1-q^{j}}-\frac{2q^{2j}}{1-q^{2j}}+\frac{2q^{4j}}{1-q^{4j}}\right)\\
&\qquad+\frac{4}{3}\frac{\eta_{2}^{3}\eta_{46}^{3}}{\eta_{1}\eta_{4}\eta_{23}\eta_{92}}-\frac{4}{3}\eta_{2}\eta_{46},\nonumber\\
\label{t23k1}
q^{3}\psi(q)\psi(q^{23})&=\frac{1}{3}\sum_{j=1}^{\infty}\left(\frac{j}{23}\right)\left(\frac{q^{j}}{1-q^{j}}-\frac{q^{2j}}{1-q^{2j}}\right)-\frac{1}{3}\eta_{1}\eta_{23}-\frac{2}{3}\eta_{2}\eta_{46}.
\end{align}
\end{Example}

The identity \eqref{s3k1} was first stated in an equivalent form by Lorenz \cite[p. 420]{lorenz}.
Both \eqref{s3k1} and \eqref{t3k1} were given by Ramanujan in his second notebook \cite[Ch. 19, Entry 3]{ramanujan1}.
See Berndt \cite[pp. 223--224]{berndt}, Fine \cite[p. 73, (31.16), (31.22)]{fine} and Hirschhorn \cite{hirschhorn} for proofs and further information.

The identities \eqref{s7k1} and \eqref{t7k1} also appear in Ramanujan's second notebook \cite[Ch. 19, Entry 17]{ramanujan1}
and proofs have been given by Berndt \cite[pp. 302--304]{berndt}.
The identity \eqref{s7k1} has also been proved by Pall \cite{pall}. 

The identities \eqref{s11k1} and \eqref{s23k1} have been recently proved by the third author \cite{dye}.

  The identities \eqref{t11k1} and \eqref{t23k1} are new.

\begin{Example}
For $p=3$, the cases $k=2,\,3,\,4$ and $6$ of Theorem $\ref{main}$ and Corollary~$\ref{main2}$ give
\begin{align}
\label{s3k2}
\varphi^{2}(q)\varphi^{2}(q^{3})&=1+4\sum_{j=1}^{\infty}\left(\frac{jq^{j}}{1-q^{j}}-\frac{2jq^{2j}}{1-q^{2j}}-\frac{3jq^{3j}}{1-q^{3j}}+\frac{4jq^{4j}}{1-q^{4j}}+\frac{6jq^{6j}}{1-q^{6j}}-\frac{12jq^{12j}}{1-q^{12j}}\right),\\
\label{t3k2}
q\psi^2(q)\psi^2(q^3) &= \sum_{j=1}^\infty \left(\frac{jq^j}{1-q^j}-\frac{jq^{2j}}{1-q^{2j}}-\frac{3jq^{3j}}{1-q^{3j}}+\frac{3jq^{6j}}{1-q^{6j}}\right), \\
\label{s3k3}
\varphi^3(q)\varphi^3(q^3)
&= 1+3\sum_{j=1}^\infty \left(\frac{j^2(q^{j}-q^{2j})}{1-q^{3j}} +  \frac{8j^2(q^{4j}-q^{8j})}{1-q^{12j}}\right) \\
&\qquad 
-\sum_{j=1}^\infty \left(\frac{j}{3}\right)\left(\frac{j^2q^j}{1-q^j}+\frac{8j^2q^{4j}}{1-q^{4j}}\right)+4\eta_2^3\eta_6^3, \nonumber \\
\label{t3k3}
q^3\psi^3(q^2)\psi^3(q^6) &= \frac{3}{32}\sum_{j=1}^\infty \left(\frac{j^2(q^{j}-q^{2j})}{1-q^{3j}} - \frac{3j^2(q^{2j}-q^{4j})}{1-q^{6j}}- \frac{4j^2(q^{4j}-q^{8j})}{1-q^{12j}}\right) \\
&\qquad -\frac{1}{32}\sum_{j=1}^\infty \left(\frac{j}{3}\right)\left(\frac{j^2q^j}{1-q^j}+\frac{3j^2q^{2j}}{1-q^{2j}}-\frac{4j^2q^{4j}}{1-q^{4j}}\right)
-\frac{1}{16}\eta_2^3\eta_6^3, \nonumber \\
\label{s3k4}
\varphi^4(q)\varphi^4(q^3) &= 1+\frac85\sum_{j=1}^\infty \left(\frac{j^3q^j}{1-q^j}-\frac{2j^3q^{2j}}{1-q^{2j}}
+\frac{9j^3q^{3j}}{1-q^{3j}}+\frac{16j^3q^{4j}}{1-q^{4j}} -\frac{18j^3q^{6j}}{1-q^{6j}}+\frac{144j^3q^{12j}}{1-q^{12j}}\right) \\
&\qquad +\frac{32}{5}\frac{\eta_{2}^{8}\eta_{6}^{8}}{\eta_{1}^{2}\eta_{3}^{2}\eta_{4}^{2}\eta_{12}^{2}}, \nonumber \\
\label{t3k4}
q^2\psi^4(q)\psi^4(q^3) &= \frac{1}{10}\sum_{j=1}^\infty \left(\frac{j^3q^j}{1-q^j}-\frac{j^3q^{2j}}{1-q^{2j}}+\frac{9j^3q^{3j}}{1-q^{3j}}-\frac{9j^3q^{6j}}{1-q^{6j}}\right) -\frac{1}{10}\eta_1^2\eta_2^2\eta_3^2\eta_6^2, \\
\label{s3k6}
\varphi^6(q)\varphi^6(q^3) &= 1+\frac{4}{13}\sum_{j=1}^\infty \left(\frac{j^5q^j}{1-q^j}-\frac{2j^5q^{2j}}{1-q^{2j}}
-\frac{27j^5q^{3j}}{1-q^{3j}}+\frac{64j^5q^{4j}}{1-q^{4j}} +\frac{54j^5q^{6j}}{1-q^{6j}}-\frac{1728j^5q^{12j}}{1-q^{12j}}\right) \\
&\qquad +\frac{152}{13}\frac{\eta_{2}^{18}\eta_{6}^{18}}{\eta_{1}^{6}\eta_{3}^{6}\eta_{4}^{6}\eta_{12}^{6}}-\frac{256}{13}\eta_2^6\eta_6^6, \nonumber \\
\label{t3k6}
q^3\psi^6(q)\psi^6(q^3) &= \frac{1}{208}\sum_{j=1}^\infty \left(\frac{j^5q^j}{1-q^j}-\frac{j^5q^{2j}}{1-q^{2j}}
-\frac{27j^5q^{3j}}{1-q^{3j}} +\frac{27j^5q^{6j}}{1-q^{6j}}\right) -\frac{1}{208}\eta_1^6\eta_3^6-\frac{19}{104}\eta_2^6\eta_6^6. 
\end{align}
\end{Example}

The identity \eqref{s3k2} was first stated without proof in an equivalent form by Liouville \cite{liouville1, liouville2}. See Pepin \cite{pepin}, Bachmann \cite{bachmann}, Kloosterman \cite{kloo} and Alaca et al. \cite{alaca1} for proofs. Both \eqref{s3k2} and \eqref{t3k2} appear in Ramanujan's second notebook
\cite[Ch. 19, Entry 3]{ramanujan1}. Proofs have been given by Fine \cite[(31.4)--(31.43) and (33.2)]{fine} and Berndt \cite[pp. 223--226]{berndt}.

A formula equivalent to \eqref{s3k3} was proved by Alaca et al. in \cite{alaca2}, where it was attributed to Berkovich and Ye$\acute{\mbox{s}}$ilyurt. 

The identity \eqref{s3k4} was proved by Alaca and Williams \cite{alaca3}. A formula similar to \eqref{s3k4},
in which the coefficients in $\eta_2^2\eta_2^2\eta_3^2\eta_6^2$ are given as a quadruple sum, has been given by Beridze \cite{beridze}.

A formula equivalent to \eqref{s3k6} was given by Alaca \cite{alaca4}; that formula involves
three cusp forms on the right hand side, while ours involves only two\footnote{This is because,
if $b(q)=\eta_1^6\eta_3^6$, then $b(q)+12b(q^2)+64b(q^4)+b(-q)=0.$}.

The identities \eqref{t3k3}, \eqref{t3k4} and \eqref{t3k6} are believed to be new.

\begin{Example} For $p=7$, the cases $k=2$ and $3$ of Theorem~$\ref{main}$ give
\begin{align}
\label{s7k2}
\varphi^2(q)\varphi^2(q^7) &= 1+\frac43\sum_{j=1}^\infty \left(\frac{jq^j}{1-q^j}-\frac{2jq^{2j}}{1-q^{2j}}+\frac{4jq^{4j}}{1-q^{4j}}-\frac{7jq^{7j}}{1-q^{7j}}+\frac{14jq^{14j}}{1-q^{14j}}-\frac{28jq^{28j}}{1-q^{28j}}\right) \\
&\qquad +\frac83\frac{\eta_{2}^{4}\eta_{14}^{4}}{\eta_{1}\eta_{4}\eta_{7}\eta_{28}}, \nonumber \\
\label{t7k2}
q^2\psi^2(q)\psi^2(q^7) &= \frac13\sum_{j=1}^\infty \left(\frac{jq^j}{1-q^j}-\frac{jq^{2j}}{1-q^{2j}}-\frac{7jq^{7j}}{1-q^{7j}}+\frac{7jq^{14j}}{1-q^{14j}}\right) \\
&\qquad -\frac13\eta_1\eta_2\eta_7\eta_{14}, \nonumber \\
\label{s7k3}
\varphi^3(q)\varphi^3(q^7)
&= 1+\frac78\sum_{j=1}^\infty \frac{j^2(q^{j}+q^{2j}-q^{3j}+q^{4j}-q^{5j}-q^{6j})}{1-q^{7j}}  \\
&\qquad -\frac{7}{4}\sum_{j=1}^\infty \frac{j^2(q^{2j}+q^{4j}-q^{6j}+q^{8j}-q^{10j}-q^{12j})}{1-q^{14j}} \nonumber \\
&\qquad + 7 \sum_{j=1}^\infty \frac{j^2(q^{4j}+q^{8j}-q^{12j}+q^{16j}-q^{20j}-q^{24j})}{1-q^{28j}}\nonumber\\
&\qquad -\frac18\sum_{j=1}^\infty \left(\frac{j}{7}\right)\left(\frac{j^2q^j}{1-q^j}-\frac{2j^2q^{2j}}{1-q^{2j}}+\frac{8j^2q^{4j}}{1-q^{4j}}\right)\nonumber\\
&\qquad+\frac{21}{4}\frac{\eta_{2}^{9}\eta_{14}^{9}}{\eta_{1}^{3}\eta_{4}^{3}\eta_{7}^{3}\eta_{28}^{3}}-6\eta_2^3\eta_{14}^3,\nonumber \\
\label{t7k3}
q^3\psi^3(q)\psi^3(q^7) &= \frac7{64}\sum_{j=1}^\infty \frac{j^2(q^{j}+q^{2j}-q^{3j}+q^{4j}-q^{5j}-q^{6j})}{1-q^{7j}}  \\
&\qquad -\frac{7}{64}\sum_{j=1}^\infty \frac{j^2(q^{2j}+q^{4j}-q^{6j}+q^{8j}-q^{10j}-q^{12j})}{1-q^{14j}} \nonumber \\
&\qquad -\frac{1}{64}\sum_{j=1}^\infty \left(\frac{j}{7}\right)\left(\frac{j^2q^j}{1-q^j}-\frac{j^2q^{2j}}{1-q^{2j}}\right)
-\frac{3}{32}\eta_1^3\eta_7^3-\frac{21}{32}\eta_2^3\eta_{14}^3. \nonumber
\end{align}
\end{Example}

Identities equivalent to \eqref{s7k2} and \eqref{t7k2} have been proved in \cite{level14}.
The identities \eqref{s7k3} and \eqref{t7k3} arise in the theory of 7-cores and
were proved by Berkovich and Yesilyurt \cite{berkovich}.

\section{Proofs}
\label{proofs}
For any positive integer $N$, let us define
$$
\Gamma_{0}(N)=\Bigg\{\begin{pmatrix}a&b\\c&d\end{pmatrix}:a,b,c,d\in\mathbb{Z}, ad-bc=1, c\equiv{0}\pmod{N}\Bigg\}.
$$
We require the explicit modularity properties of the Eisenstein series on $\Gamma_0(p)$ as well as the Atkin--Lehner involution $W_p$.
\begin{Lemma}
\label{lemma1}
For $p=3$, $7$, $11$ or $23$, and any integer $k\geq0$, and for any $\begin{pmatrix}a&b\\c&d\end{pmatrix}\in\Gamma_{0}(p)$, we have
\begin{align}
E_{2k+1}^{0}\left(\frac{a\tau+b}{c\tau+d};\chi_{p}\right)&=\left(\frac{d}{p}\right)(c\tau+d)^{2k+1}E_{2k+1}^{0}(\tau;\chi_{p}),\\
E_{2k+1}^{\infty}\left(\frac{a\tau+b}{c\tau+d};\chi_{p}\right)&=\left(\frac{d}{p}\right)(c\tau+d)^{2k+1}E_{2k+1}^{\infty}(\tau;\chi_{p}),\\
\label{Eptrans1}
E_{2k+1}^{\infty}\left(-\frac{1}{p\tau};\chi_{p}\right)&=\frac{1}{i\sqrt{p}}(p\tau)^{2k+1}E_{2k+1}^{0}\left(\tau;\chi_{p}\right),\\
\label{Eptrans2}
E_{2k+1}^{0}\left(-\frac{1}{p\tau};\chi_{p}\right)&=\frac{\sqrt{p}}{i}\tau^{2k+1}E_{2k+1}^{\infty}\left(\tau;\chi_{p}\right).
\end{align}
\end{Lemma}

\begin{proof}
This is a well-known result, e.g., see Cooper \cite{cooper} or Kolberg \cite{kolberg}.
\end{proof}

\begin{Lemma}
\label{F1/2}
For any positive integer $k$, we have
\begin{align}
\label{E2kvan}
E_{2k}\left(\tau+\frac{1}{2}\right)=-E_{2k}(\tau)+(2^{2k}+2)E_{2k}(2\tau)-2^{2k}E_{2k}(4\tau).
\end{align}
For $p=3$ or $11$ and any nonnegative integer $k$, we have
\begin{align}
\label{E2k/2}
E_{2k+1}^{0}\left(\tau+\frac{1}{2};\chi_{p}\right)&=-E_{2k+1}^{0}(\tau;\chi_{p})+(2^{2k+1}-2)E_{2k+1}^{0}(2\tau;\chi_{p})+2^{2k+1}E_{2k+1}^{0}(4\tau;\chi_{p}),\\
\label{E2k3}
E_{2k+1}^{\infty}\left(\tau+\frac{1}{2};\chi_{p}\right)&=-E_{2k+1}^{\infty}(\tau;\chi_{p})+(2-2^{2k+1})E_{2k+1}^{\infty}(2\tau;\chi_{p})+2^{2k+1}E_{2k+1}^{\infty}(4\tau;\chi_{p}).
\end{align}
For $p=7$ or $23$ and any nonnegative integer $k$, we have
\begin{align}
E_{2k+1}^{0}\left(\tau+\frac{1}{2};\chi_{p}\right)&=-E_{2k+1}^{0}(\tau;\chi_{p})+(2^{2k+1}+2)E_{2k+1}^{0}(2\tau;\chi_{p})-2^{2k+1}E_{2k+1}^{0}(4\tau;\chi_{p}),\\
\label{E7half}
E_{2k+1}^{\infty}\left(\tau+\frac{1}{2};\chi_{p}\right)&=-E_{2k+1}^{\infty}(\tau;\chi_{p})+(2^{2k+1}+2)E_{2k+1}^{\infty}(2\tau;\chi_{p})-2^{2k+1}E_{2k+1}^{\infty}(4\tau;\chi_{p}).
\end{align}
\end{Lemma}

\begin{proof}
First of all, from the definitions of $E_{2k}(\tau)$, $E_{2k+1}^{0}(\tau;\chi_{p})$ and $E_{2k+1}^{\infty}(\tau;\chi_{p})$,
we may deduce the Fourier expansions
\begin{align}
E_{2k}(\tau)&=1-\frac{4k}{\mathcal{B}_{2k}}\sum_{n=1}^{\infty}\sigma_{2k-1}(n)q^{n},\\
E_{2k+1}^{0}(\tau;\chi_{p})&=\delta_{2k+1,1}-\frac{4k+2}{\mathcal{B}_{2k+1,p}}\sum_{n=1}^{\infty}\left(\sum_{d|n}\left(\frac{n/d}{p}\right)d^{2k}\right)q^{n}
\intertext{and}
E_{2k+1}^{\infty}(\tau;\chi_{p})&=1-\frac{4k+2}{\mathcal{B}_{2k+1,p}}\sum_{n=1}^{\infty}\left(\sum_{d|n}\left(\frac{d}{p}\right)d^{2k}\right)q^{n},
\end{align}
where $\sigma_k(n) = \sum_{d|n} d^k$ if $n$ is a positive integer.
For \eqref{E2kvan}, we first observe that
$$
\sigma_{k}(2n)=\left(1+2^{k}\right)\sigma_{k}(n)-2^{k}\sigma_{k}\left(n/2\right),
$$
where $\sigma_k(n/2)$ is defined to be zero
if $n/2$ is not a positive integer.
Then,
\begin{align*}
&E_{2k}\left(\tau+\frac{1}{2}\right)+E_{2k}(\tau)\\
&=2\left(1-\frac{4k}{\mathcal{B}_{2k}}\sum_{n=1}^{\infty}\sigma_{2k-1}(2n)q^{2n}\right)\\
&=2\left(1-\frac{4k}{\mathcal{B}_{2k}}\sum_{n=1}^{\infty}\left((1+2^{2k-1})\sigma_{2k-1}(n)-2^{2k-1}\sigma_{2k-1}\left(n/2\right)\right)q^{2n}\right)\\
&=2\left(\left(1+2^{2k-1}\right)\left(1-\frac{4k}{\mathcal{B}_{2k}}\sum_{n=1}^{\infty}\sigma_{2k-1}(n)q^{2n}\right)-2^{2k-1}\left(1-\frac{4k}{\mathcal{B}_{2k}}\sum_{n=1}^{\infty}\sigma_{2k-1}(n)q^{4n}\right)\right)\\
&=\left(2+2^{2k}\right)E_{2k}(2\tau)-2^{2k}E_{2k}(4\tau).
\end{align*}
If we let $\sigma_{2k+1}^{0}(n)$ and $\sigma_{2k+1}^{\infty}(n)$ be defined by
$$
\sigma_{2k+1}^{0}(n)=\sum_{d|n}\left(\frac{n/d}{p}\right)d^{2k},\quad\mbox{and}\quad \sigma_{2k+1}^{\infty}(n)=\sum_{d|n}\left(\frac{d}{p}\right)d^{2k},
$$
then similarly, we find that
\begin{align*}
\sigma_{2k+1}^{0}(2n)&=\begin{cases}(2^{2k}-1)\sigma_{2k+1}^{0}(n)+2^{2k}\sigma_{2k+1}^{0}(n/2),&\mbox{for $p=3$ or $11$,}\\
(1+2^{2k})\sigma_{2k+1}^{0}(n)-2^{2k}\sigma_{2k+1}^{0}(n/2),&\mbox{for $p=7$ or $23$,}\end{cases}
\intertext{and}
\sigma_{2k+1}^{\infty}(2n)&=\begin{cases}(1-2^{2k})\sigma_{2k+1}^{\infty}(n)+2^{2k}\sigma_{2k+1}^{\infty}(n/2),&\mbox{for $p=3$ or $11$,}\\
(1+2^{2k})\sigma_{2k+1}^{\infty}(n)-2^{2k}\sigma_{2k+1}^{\infty}(n/2),&\mbox{for $p=7$ or $23$.}\end{cases}
\end{align*}
By the above observations, identities \eqref{E2k/2}--\eqref{E7half} can be proved in the same fashion, so we omit the details.
\end{proof}

\begin{Lemma}
\label{Fhalf}
For any positive integer $k$, we have
\begin{equation}
\label{F2k}
F_{2k}\left(-\frac{1}{4p\tau}+\frac{1}{2};p\right) =
\frac{\tau^{2k}(-16p)^k\left( G_{2k}(2\tau)-G_{2k}(4\tau)\right)}
{(2^{2k}-1)(1+(-p)^k)}.
\end{equation}
For $p=3$ or $11$ and any non-negative integer $k$, we have
\begin{align}
\label{Fodd3}
&F_{2k+1}\left(-\frac{1}{4p\tau}+\frac{1}{2};p\right)\\
&=\left(2i\tau\sqrt{p}\right)^{2k+1}
\times \frac{\left(2^{2k+1}G_{2k+1}(4\tau;p) - 2(2^{2k}-1)\tilde{G}_{2k+1}(2\tau;p)
-2G_{2k+1}(\tau;p)\right)}{(2^{2k+1}+1)(1+\delta_{k,0})}.\nonumber
\end{align}
For $p=7$ or $23$ and any non-negative integer $k$, we have
\begin{align}
\label{Fodd7}
&F_{2k+1}\left(-\frac{1}{4p\tau}+\frac{1}{2};p\right)=\left(4i\tau\sqrt{p}\right)^{2k+1}\times\frac{\left({G}_{2k+1}(4\tau;p)-G_{2k+1}(2\tau;p)\right)}{(2^{2k+1}-1)(1+\delta_{k,0})}.
\end{align}

\end{Lemma}

\begin{proof}
These follow immediately from the definitions of $F_{k}(\tau;p)$ together with Lemmas \ref{lemma1} and \ref{F1/2} 
and the weight $2k$ modularity of $E_{2k}$ on $\SL_2(\Z)=\Gamma_0(1)$.

\end{proof}

Now we are ready for
\begin{proof}[Proof of Theorem $\ref{main}$]

Let $p=3$, $7$, $11$ or $23$, and let $k$ be a positive integer. Let $\ell$ be the smallest integer that satisfies 
$$
\ell\geq\begin{cases}\frac{(p+1)k}{8}-\frac{1}{2},&\mbox{if $p=3$ or $11$ and $k$ is odd,}\\
\frac{(p+1)k}{8}-1,&\mbox{otherwise.}\end{cases}
$$
Consider the functions
$$
f(\tau)=f_{k,p}(\tau)=\frac{F_{k}(\tau;p)}{z_{p}(\tau)^{k}x_{p}(\tau)^{\ell}}\quad\mbox{and}\quad
g(\tau)=g_{p}(\tau)=\frac{1}{x_{p}(\tau)}.
$$
Clearly, both $f(\tau)$ and $g(\tau)$ are analytic on $\mathbb{H}$. {And we may verify that both $f(\tau)$ and $g(\tau)$ are invariant under $\Gamma_{0}(4p)$ and 
$$
W_{e}=\left\{\begin{pmatrix}ae&b\\4pc&de\end{pmatrix}:a,\,b,\,c,\,d\in\mathbb{Z},\,\mbox{the determinant is $e$}\right\}
$$
for $e\in\{4,\,p,\,4p\}$. Therefore, both $f(\tau)$ and $g(\tau)$ are invariant under $
{{
\Gamma_{0}(4p)^+
}}
$, the group obtained from $\Gamma_{0}(4p)$ by adjoining all of its Atkin-Lehner involutions $W_{e}$.}  Let us analyze the behavior at $\tau=i\infty$. By observing the $q$-expansions, we find that 
$$
f(\tau)=\frac{1+O(q)}{(1+O(q))^{k}q^{\ell}(1+O(q))^{\ell}}=q^{-\ell}+O(q^{-\ell+1}).
$$
Therefore $f(\tau)$ has a pole of order $\ell$ at $i\infty$. Similarly, we note that $g(\tau)$ has a simple pole at $\tau=i\infty$. It implies that there exist 
constants
 $a_{1},\,\ldots,a_{\ell}\in \mathbb{C}$ such that the function
$$
h(\tau):=f(\tau)-\sum_{j=1}^{\ell}a_{j}g(\tau)^{j}
$$
has no pole at $\tau=i\infty$, that is,
$$
h(\tau)=a_{0}+O(q)\quad\mbox{as $\tau\to i\infty$}
$$
for some constant $a_{0}$. Let us consider the behavior of $h(\tau)$ at $\tau=\frac{1}{2}$. By Lemma \ref{Fhalf} and the transformation formula for Dedekind's eta function, we find that
$$
f\left(-\frac{1}{4p\tau}+\frac{1}{2}\right)=\begin{cases}C_{0}q^{2\ell-\frac{(p+1)k}{4}+1}(1+O(q)),&\mbox{if $p=3$ or $11$ and $k$ is odd,}\\ C_{1}q^{2\ell-\frac{(p+1)k}{4}+2}(1+O(q)),&\mbox{otherwise,}\end{cases}
$$
and
$$
g\left(-\frac{1}{4p\tau}+\frac{1}{2}\right)=C_{3}q^{2}(1+O(q))
$$
for some constants $C_{0}$, $C_{1}$ and $C_{2}$ as $\tau\to i\infty$. Therefore, $h(\tau)\to a_{0}$ as $\tau\to\frac{1}{2}$. 
{{Since the only cusps of $
{{
\Gamma_{0}(4p)^+
}}
$ are at $i\infty$ and $\frac{1}{2}$, it follows that $h(\tau)$}} is
 holomorphic on $X(
{{
\Gamma_{0}(4p)^+
}}
)$, and thus $h(\tau)$ is a constant, that is, $h(\tau)\equiv a_{0}$. Therefore, we have
$$
f(\tau)=\sum_{j=0}^{\ell}a_{j}g(\tau)^{j},
$$
which is equivalent to
$$
F_{k}(\tau;p)=z_{p}^{k}\sum_{j=0}^{\ell}a_{j}x_{p}^{\ell-j}=z_{p}^{k}\sum_{j=0}^{\ell}b_{j}x_{p}^{j},
$$
where $b_{j}:=a_{\ell-j}$. By the choice of $\ell$ and comparing the constant terms on both sides, we conclude that
{{$b_0=1$ and}}
$$
F_{k}(\tau;p)=z_{p}^{k}+z_{p}^{k}\sum_{1\leq j<\frac{(p+1)k}{8}}b_{j}x_{p}^{j}.
$$
{{Now take $b_j=-c_{p,k,j}$ to complete the proof.}}
\end{proof}

\begin{proof}[Proof of Corollary $\ref{main2}$]
It follows directly from Theorem $\ref{main}$ together with the following observations:
\begin{align}
\tilde{z}_{p}(\tau)&=\frac{i}{4\sqrt{p}\tau}z_{p}\left(-\frac{1}{2p\tau}+\frac{1}{2}\right), \nonumber \\
\tilde{x}_{p}(\tau)&=-2^{-24/(p+1)}\tilde{x}_{p}\left(-\frac{1}{2p\tau}+\frac{1}{2}\right),\nonumber \\
\tilde{F}_{k}(\tau;p)&=\left(\frac{i}{4\sqrt{p}\tau}\right)^{k}F_{k}\left(-\frac{1}{2p\tau}+\frac{1}{2}\right). \label{last}
\end{align}
\end{proof}


\begin{thebibliography}{99}

\bibitem{alaca4}
{A. Alaca, 
{\em On the number of representations of a positive integer by certain quadratic forms in twelve variables,} J. Combin. Number Theory, {\bf3} (2011), 167--177.}

\bibitem{alaca1}
{A. Alaca, S. Alaca, M. F. Lemire and K. S. Williams, 
{\em Nineteen quaternary quadratic forms,} Acta Arith. {\bf130} (2007), 277--310.}

\bibitem{alaca2}
{A. Alaca, S. Alaca and K. S. Williams, 
{\em Some new theta function identities with applications to sextenary quadratic forms,} J. Combin. Number Theory, {\bf1} (2009), 89--98.}

\bibitem{alaca3}
{S. Alaca and K. S. Williams, 
{\em The number of representations of a positive integer by certain octonary quadratic forms,} Functiones et Approximatio, {\bf43} (2010), 45--54.}

\bibitem{bachmann}
{P. Bachmann, 
{\em Niedere Zahlentheorie,} Chelsea, New York, 1968.}

\bibitem{beridze}
{R. I. Beridze, 
{\em The representation of numbers by certain quadratic forms in eight variables,} Thbilis. Univ. $\breve{\mbox{S}}$rom. A, {\bf1} (1971), 5--16.}

\bibitem{berkovich}
{A. Berkovich and H. Yesilyurt, 
{\em On the representations of integers by the sextenary quadratic form $x^{2}+y^{2}+z^{2}+7(s^{2}+t^{2}+u^{2})$ and $7$-cores,} J. Number Theory, {\bf129} (2009), 1366--1378.}

\bibitem{berndt}
{B. C. Berndt, 
{\em Ramanujan's Notebooks, Part III,} Springer-Verlag, New York, 1991.}

\bibitem{chancooper}
H. H. Chan and S. Cooper, 
{\em Powers of theta functions,} Pacific J. Math., {\bf235} (2008), 1--14.

\bibitem{sc}{S. Cooper, \emph{On sums of an even number of squares, and an even number of triangular
numbers: an elementary approach based on Ramanujan's
${}_1\psi_1$ summation formula.}
In: {\em $q$-series with applications to combinatorics, number theory, and physics}
(Urbana, IL, 2000), 115--137, Contemp. Math., {\bf 291}, Amer. Math. Soc., Providence, RI, 2001.}


\bibitem{cooper}{S. Cooper, \emph{Construction of Eisenstein series for $\Gamma_{0}(p)$}, Int. J. Number Theory {\bf5} (2009), 765--778.}

\bibitem{level14}
S. Cooper and D. Ye,
{\em Level 14 and 15 analogues of Ramanujan's elliptic functions to alternative bases,} Trans. Amer. Math. Soc., DOI: http://dx.doi.org/10.1090/tran6658

\bibitem{fine}
N. Fine,
{\em Basic Hypergeometric Series and Applications,}
AMS, Providence, RI, 1988.

\bibitem{hirschhorn}
{M. D. Hirschhorn, 
{\em Three classical results on representations of a number,} S$\acute{\mbox{e}}$m. Lothar. Combin. {\bf42} (1999), art. B42f, 8 pp.}

\bibitem{kloo}
{H. D. Kloosterman, 
{\em On the representation of numbers in the form $ax^{2}+by^{2}+cz^{2}+dt^{2}$,} Proc. London Math. Soc. {\bf25} (1926), 143--173.}


\bibitem{kolberg}{O. Kolberg, \emph{Note on the Eisenstein series of $\Gamma_{0}(p)$}, Arbok Univ. Bergen Mat.-Natur. Ser. {\bf1968} (1968), 20 pp. (1969).}

\bibitem{liouville1}
{J. Liouville, 
{\em Sur la forme $x^{2}+y^{2}+3(z^{2}+t^{2})$,} J. Math. Pures Appl. {\bf5} (1860), 147--152.}

\bibitem{liouville2}
{J. Liouville, 
{Remarque nouvelle sur la forme $x^{2}+y^{2}+3(z^{2}+t^{2})$,} ibid. {\bf6} (1861), 296.}

\bibitem{lorenz}
{L. Lorenz, 
Oeuvres Scientifiques, Revues et annot$\acute{\mbox{e}}$es par H. Valentiner, Tome second, la Fondation Carlsberg, Librarie Lehmann $\&$ Stage, Copenhague, 1904.}


\bibitem{mordell}{L. J. Mordell, \emph{On the representation of numbers as the sum of $2$r squares,} Quart. J. Pure and Appl. Math., Oxford {\bf48} (1917), 93--104.}

\bibitem{pall}
{G. Pall, 
{\em On the application of a theta formula to representation in binary quadratic forms,} Bull. Amer. Math. Soc., {\bf37} (1931), 863--869.}

\bibitem{pepin}
{T. Pepin, 
{\em Sur quelques formes quadratiques quaternaires,} J. Math. Pures Appl. {\bf6} (1890), 5--67.}

\bibitem{ramanujan1}
{S. Ramanujan, 
{\em Notebooks (2 volumes)}, Tata Institute of Fundamental Research, Bombay, 1957.}

\bibitem{ram}{S. Ramanujan, \emph{Collected Papers,} AMS Chelsea Publishing, Providence, Rhode Island, 2000.}

\bibitem{dye}
{D. Ye, {\em Representations of certain binary quadratic forms as a sum of Lambert series and eta-quotients},
Int. J. Number Theory {\bf11} (2015), 1073-1088}

\end{thebibliography}
\end{document}